\documentclass[12pt]{amsart}
\usepackage{amsmath}
\usepackage{amsthm}
\usepackage{amssymb}
\usepackage{enumerate}
\usepackage{xcolor}
\usepackage{comment}
\newtheorem{theorem}{Theorem}[section]
\newtheorem{lemma}[theorem]{Lemma}
\newtheorem{proposition}[theorem]{Proposition}

\newtheorem{remar}[theorem]{Remark}
\theoremstyle{definition}
\newtheorem{examp}[theorem]{Example}
\newtheorem{prob}[theorem]{Open Problem}
\newenvironment{example}{\begin{examp}\rm}{\diams\end{examp}}
\newcommand{\diams}{\unskip\nobreak\hfil\penalty50%
\hskip1em\hbox{}\nobreak\hfil%
$\diamondsuit$\parfillskip=0pt\finalhyphendemerits=0}

\newenvironment{remark}{\begin{remar}\rm}{\end{remar}}

\newcommand{\bfind}[1]{\index{#1}{\bf #1}}

\newcommand{\sn}{\par\smallskip\noindent}
\newcommand{\mn}{\par\medskip\noindent}
\newcommand{\bn}{\par\bigskip\noindent}
\newcommand{\pars}{\par\smallskip}

\newcommand{\chara}{\mbox{\rm char}\,}

\newcommand{\cO}{\mathcal{O}}

\newcommand{\Q}{\mathbb Q}

\newcommand{\F}{\mathbb F}

\begin{document}
\title[]{On algebraically maximal valued fields that are not defectless}
\author{Franz-Viktor Kuhlmann}
\date{25.5.2026}

%\begin{comment}

\address{Institute of Mathematics, University of Szczecin,
ul. Wielkopolska 15,
70-451 Szczecin, Poland}
\email{fvk@usz.edu.pl}

\begin{abstract}\noindent
An example originally given by F.~Delon shows the existence of an algebraically maximal
discretely valued field of characteristic $p>0$ which admits purely inseparable
extensions of degree $p^2$ with defect $p$. These extensions are not generated by a
single element. Using a trick introduced in an earlier paper of the author, we construct
algebraically maximal valued fields, of characteristic $p$ as well as of characteristic
$0$, which admit separable extensions of degree $p^2$ with defect $p$. They are of rank
2 and it is an open question whether such examples having rank 1 exist.
\end{abstract}
%\thanks{}

%\end{comment}

\maketitle

%
%------------------------------------------------------------------------------
%
\section{Introduction}
The notions and notations we will use will be introduced in Section~\ref{sectprel}.
\pars
Fran\c{c}oise Delon gave an example that shows that algebraically maximal valued
fields are not necessarily defectless (see \cite{D}, Exemple~1.51). A corrected and
expanded version was presented in \cite[Example 3.25]{Ku31}. We reproduce it in
Section~\ref{sectbex} and fill a gap that appeared in its exposition in \cite{Ku31}.

\pars
For what follows, take a prime $p$. Example~\ref{exdelon} proves:
\begin{theorem}                           \label{amnotidl}
There exists a discretely valued algebraically maximal field $(L,v)$ of characteristic
$p>0$ which is not inseparably defectless and admits a purely inseparable
extension of degree $p$ which is not an algebraically maximal field. In particular,
the property ``algebraically maximal'' does not imply ``defectless''.
\end{theorem}

\pars
The question arises whether there are also examples of algebraically maximal fields
which admit separable (and hence simple) defect extensions. Using a trick already
employed in \cite[Example 3.18]{Ku31}, we will construct such examples in
Section~\ref{sectexrk2}, based on which we prove:
\begin{theorem}                           \label{amnotidl2}
There exists an algebraically maximal field $(L,v)$ with residue characteristic
$p>0$ which admits a separable extension of degree $p^2$ with defect
$p$, and with intermediate fields of degree $p$ over $L$ which are not algebraically
maximal fields. The field $L$ can be chosen of characteristic $0$ as well as
of characteristic $p$.
In particular, the property ``algebraically maximal'' does not imply
``separably defectless'' and is not preserved under finite separable extensions.
\end{theorem}

The valuations in the examples we give to prove this theorem have rank 2.
\sn
{\bf Open Problem:} Are there algebraically maximal fields of rank 1
which admit separable defect extensions?

\mn
%
%------------------------------------------------------------------------------
%
\section{Preliminaries}                    \label{sectprel}
For a valued field $(K,v)$, we denote its value group by $vK$, its residue field by $Kv$, and its valuation ring by $\cO_K\,$.   % with maximal ideal  $\cM_K\,$.
By $(L|K,v)$ we denote an extension $L|K$ with valuation $v$ on $L$, where $K$
is endowed with the restriction of $v$. In this case, there are induced embeddings of
$vK$ in $vL$ and of $Kv$ in $Lv$. The extension $(L|K,v)$ is called \bfind{immediate}
if these embeddings are onto. A valued field $(K,v)$ is called \bfind{algebraically
maximal} if it does not admit nontrivial immediate algebraic extensions, and it is
called \bfind{maximal} if it does not admit any nontrivial immediate extensions.

If $K$ is a field of characteristic $p>0$ and the extension $K|K^p$ is finite, then
there is $k\geq 0$ such that $[K:K^p]=p^k$; we then take the \bfind{$p$-degree of $K$}
(also called \bfind{degree of inseparability}) to be $k$.
%If $K|K^p$ is infinite, then we take the $p$-degree to be $\infty$.

A valued field $(K,v)$ is called \bfind{henselian} if each algebraic extension
$L|K$ is \bfind{unibranched}, that is, the extension of $v$ to $L$ is unique.

If $(L|K,v)$ is a finite unibranched extension, then by the Lemma of Ostrowski
(\cite[Corollary to Theorem 25, Section G, p.\ 78]{ZS}),
\begin{equation}                    \label{feuniq}
[L:K]\>=\> \tilde{p}^{\nu }\cdot(vL:vK)[Lv:Kv]\>,
\end{equation}
where $\nu$ is a nonnegative integer and $\tilde{p}$ the
\bfind{characteristic exponent} of $Kv$, that is, $\tilde{p}=\chara Kv$ if it is
positive and $\tilde{p}=1$ otherwise. The factor $d(L|K,v):=\tilde{p}^{\nu }$ is
the \bfind{defect} of the extension $(L|K,v)$.
If $d(L|K,v)=1$, then the extension $(L|K,v)$ is called \bfind{defectless};
otherwise we call it a \bfind{defect extension}. A valued field $(K,v)$ is a
\bfind{separably defectless field} if every finite unibranched separable extension
of $(K,v)$ is defectless, and a \bfind{defectless field} if every finite unibranched
extension of $(K,v)$ is defectless; note that this is always the case if
$\chara Kv=0$. An arbitrary valued field is called an \bfind{inseparably defectless
field} if every finite purely inseparable extension is defectless.

The defect is multiplicative: if $(L|K,v)$ and
$(M|L,v)$ are finite unibranched extensions, then
\begin{equation}         \label{pf}
\mbox{\rm d}(M|K,v) = \mbox{\rm d}(M|L,v)\cdot\mbox{\rm d}(L|K,v)
\end{equation}
(see \cite[Equation (4)]{Ku31}).

For a valued field $(K,v)$ and a finite field extension $L|K$, the \bfind{Fundamental
Inequality} (see (17.5) of \cite{End} or Theorem 19 on p.~55 of \cite{ZS}) states
that there are finitely many extensions of $v$ from $K$ to $L$, and
\begin{equation}                             \label{fundineq}
[L:K]\>\geq\>\sum_{i=1}^{\rm g} (v_i L:vK)[Lv_i:Kv]\>,
\end{equation}
where $v_1,\ldots,v_{\rm g}$ are the distinct extensions.

\pars
For a field $K$ of characteristic $p>0$, we will denote its perfect hull by
$K^{1/p^{\infty}}$. An arbitrary field extension $L$ of $K$ is called \bfind{separable}
if it is linearly disjoint over $K$ from $K^{1/p^{\infty}}$. If $L|K$ is algebraic, then
this definition coincides with the usual definitions of separability.
\begin{lemma}                               \label{charsep}
An arbitrary field extension $L$ of a field $K$ of characteristic $p>0$ is separable if
and only if it is linearly disjoint over $K$ from $K^{1/p}=K(a^{1/p}\mid a\in K)$.
\end{lemma}
\begin{proof}
If the extension $L$ of $K$ is separable, then by definition it is linearly disjoint
over $K$ from $K^{1/p^{\infty}}$ and thus also from $K^{1/p}$. For the converse,
assume that $L$ is not linearly disjoint over $K$ from $K^{1/p^{\infty}}$.
Then there are $K$-linearly independent elements $x_1,\ldots,x_n\in L$ which are not
$K^{1/p^{\infty}}$-linearly independent. Choose $m$ minimal such that there are
$y_1,\ldots,y_n\in K^{1/p^m}=K(a^{1/p^m}\mid a\in K)$ with $\sum_{i}
x_iy_i=0$. Then $m\geq 1$, and $x_1^{p^{m-1}},\ldots,x_n^{p^{m-1}}$ are
$K$-linearly independent. (Otherwise, we would have a non-trivial
relation $\sum_{i} x_i^{p^{m-1}} z_i=0$ with $z_i\in K$, hence $\sum_{i}
x_i z_i^{1/p^{m-1}}=0$, contradicting the minimality of $m$.)
But $x_1^{p^{m-1}},\ldots,x_n^{p^{m-1}}$ are not $K^{1/p}$-linearly
independent since $\sum_{i} x_i^{p^{m-1}} y_i^{p^{m-1}}=0$ with
$y_i^{p^{m-1}}\in K^{1/p}$. This proves that $L|K$ is not linearly
disjoint over $K$ from $K^{1/p}|K$.
\end{proof}

%
%------------------------------------------------------------------------------
%
\section{A basic example}                    \label{sectbex}
\begin{example}                             \label{exdelon}
We consider $\F_p((t))$ with its $t$-adic valuation $v_t\,$. Since $\F_p((t))$ has
uncountable cardinality, while that of $\F_p(t)$ is countable, the extension
$\F_p((t))|\F_p(t)$ has infinite transcendence degree, we can choose elements
$x,y\in\F_p((t))$ which are algebraically independent over $\F_p(t)$. We set
\[
s\>:=\>x^p+ty^p\;\mbox{ \ and \ }\; K\>:=\>\F_p(t,s)\>.
\]
We note that $K^{1/p}=\F_p(t^{1/p},s^{1/p})=K(t^{1/p},s^{1/p})$.
The elements $t,s$ are algebraically independent over $\F_p$.
Consequently, the $p$-degree of $K$ is $2$. We define
$L_0$ to be the relative algebraic closure of $K$ in $\F_p((t))$. Then
$L_0(t^{1/p},s^{1/p})\subseteq L_0^{1/p}$.

We are going to show that $L_0$ is linearly disjoint over $K$ from $K^{1/p}$ and the
$p$-degree of $L_0$ is again $2$. Since the elements $1,t^{1/p},
\ldots, t^{(p-1)/p}$ are linearly independent over $\F_p((t))$, the same
holds over $L_0\,$. Hence, the elements $1,t,\ldots,t^{p-1}$ are linearly
independent over $L_0^p$. Now if $s^{1/p}$ were an element of $L_0(t^{1/p})$,
then it could be written in a unique way as an $L_0$-linear combination of
$1,t^{1/p},\ldots, t^{(p-1)/p}$, and $s$ could be
written in a unique way as an $L_0^p$-linear combination of $1,t,\ldots,
t^{p-1}$. But this is not possible since $s=x^p+ty^p$ and $x,y$ are
transcendental over $L_0\,$. Hence, $L_0$ is linearly disjoint over $K$ from $K^{1/p}$
and the $p$-degree of $L_0$ is again $2$; more precisely,
\[
L_0^{1/p}\>=\>L_0(t^{1/p},s^{1/p})\quad \mbox{ with } \quad [L_0^{1/p}:L_0]
\>=\>p^2\>.
\]
From Lemma~\ref{charsep} it follows that the extension $L_0|K$ is separable.

Since $s^{1/p}=x+t^{1/p}y\in \F_p(t^{1/p},x,y)\subset\F_p((t^{1/p}))$,
we have $L_0^{1/p}\subset \F_p((t^{1/p}))$. Extending $v_t$ to $\F_p((t^{1/p}))$, we
obtain that
\[
L_0^{1/p}v_t\>\subseteq\> \F_p((t^{1/p}))v_t\>=\>\F_p\>=\>Kv_t\>\subseteq\>
L_0v_t\>\subseteq\> L_0^{1/p}v_t\>.
\]
Further,
\[
v_t L_0^{1/p}\>\subseteq\> v_t\F_p((t^{1/p}))\>=\>\frac 1 p v_t \F_p((t))\>=\>
v_t \F_p(t^{1/p})\>\subseteq\> v_t L_0^{1/p}\>,
\]
hence equality holds everywhere. Moreover, $v_t L_0\subseteq
v_t \F_p((t))=v_t \F_p(t)\subseteq v_t L_0\,$, showing that $v_t L_0=v_t \F_p((t))$.
Consequently,
\[
(v_tL_0^{1/p}:v_tL_0)\>=\>p\quad \mbox{ and } \quad [L_0^{1/p}v_t:L_0v_t]\>=\>1
\]
As a relatively algebraically closed subfield of the henselian field $(\F_p((t)),
v_t)$, also $(L_0,v_t)$ is henselian. Thus the extension $(L_0^{1/p}|L_0,v_t)$ is
unibranched and consequently has defect $p$.

\pars
On the other hand, $\F_p((t))$ is the completion of $(L_0,v_t)$ since it is
already the completion of $\F_p(t)\subseteq L_0$. This shows that
$\F_p((t))$ is the unique maximal immediate extension of $L_0$ (up to
valuation preserving isomorphism over $L_0$). If $L_0$ would admit a proper
immediate algebraic extension $L_1$, then a maximal immediate extension
of $L_1$ would also be a maximal immediate extension of $L_0$ and would
thus be isomorphic over $L_0$ to $\F_p((t))$. But we have chosen $L_0$ to be
relatively algebraically closed in $\F_p((t))$. This proves that $(L_0,v)$
must be algebraically maximal. Hence
\[
(L_0(s^{1/p})|L_0,v_t)\quad \mbox{ and } \quad (L_0(t^{1/p})|L_0,v_t)\>,
\]
cannot be immediate and, being of prime degree, are therefore defectless. Thus the
defect of $L_0^{1/p}|L_0$ implies by multiplicativity (\ref{pf}) that both
\[
(L_0^{1/p}|L_0(s^{1/p}),v_t)\quad \mbox{ and } \quad (L_0^{1/p}|L_0(t^{1/p}),v_t)
\]
must have defect $p$. Consequently, $(L_0(s^{1/p},v_t)$ and $(L_0(t^{1/p},v_t)$
are not algebraically maximal.
\end{example}

\pars
We summarize the properties of this example, thereby adjusting the notation for
later use.
\begin{proposition}                            \label{L_0v_0}
There exists a discretely valued algebraically maximal field $(L_0,v_0)$ of
characteristic $p>0$ and purely inseparable defectless extensions $(L_0(a_0)|L_0,
v_0)$ and $(L_0(b_0)|L_0,v_0)$ of degree $p$ such that the unibranched extension
$(L_0(a_0,b_0)|L_0,v_0)$ of degree $p^2$ has defect $p$, as
$(v_0L_0(a_0,b_0):v_0L_0)=p$
and $[L_0(a_0,b_0)v_0:L_0v_0]=1$, and neither $(L_0(a_0),v_0)$ nor $(L_0(b_0),
v_0)$ is an algebraically maximal field. \qed
\end{proposition}

\begin{remark}
In \cite[Example 3.25]{Ku31} it is stated that the relative algebraic closure $L_0$
of $(K,v_t)$ in $\F_p((t))$ is a separable extension of $K$ and therefore is the
henselization of $K$. However, the proof on page 295 of \cite{Ku31} contains a gap
(but this does not affect the results on the previous pages). In fact, it turns out
that the use of the auxiliary field $\F_p(t,x,y)$ which contains $K$ is not
even necessary. Indeed, we have shown that the extension $L_0|K$ is separable.
Since by definition, $L_0$ is relatively closed in the henselian field
$\F_p((t))$, it is itself henselian and thus contains the henselization
$K^h$ of $K$. Now $\F_p((t))$ is the completion of $K^h$ since it is
already the completion of $\F_p(t)\subseteq K^h$. Since a henselian
field is relatively separable-algebraically closed in its completion
(cf.\ \cite{W}, Theorem~32.19) and $L_0|K^h$ is separable, it follows that $L_0=K^h$.

Nevertheless, the field $F:=\F_p(t,x,y)$ is useful for showing that the extension
$\F_p((t))|L_0$ is not separable. We have $s^{1/p}=x+t^{1/p}y\in F(t^{1/p})$.
Hence, $F.K^{1/p}=F(t^{1/p},s^{1/p})=F(t^{1/p})$ and $[F.K^{1/p}:F]=
[F(t^{1/p}):F] \leq p<p^2=[K^{1/p}:K]$, that is, $F$ is not linearly
disjoint over $K$ from $K^{1/p}$ and thus not separable, according to
Lemma~\ref{charsep}. Since $F\subset\F_p((t))$, it follows that $\F_p((t))|K$ is
not separable. Since $L_0|K$ is separable, this implies that $\F_p((t))|L_0$ is
not separable.
\end{remark}

\mn
%
%------------------------------------------------------------------------------
%
\section{Examples with composite valuations}              \label{sectexrk2}

\begin{lemma}                                 \label{L,w}
Take any field $L_0$ of positive characteristic. There exist henselian defectless
discretely valued fields $(L,w)$ with residue field $L_0\,$.
They can be chosen such that either $\chara L=0$, or $\chara L=\chara L_0\,$.
\end{lemma}
\begin{proof}
For $\chara L=0$: Take an extension of $(\Q,v_p)$, where $v_p$ denotes the $p$-adic
valuation, with value group equal to $v_p\Q$ and residue field $L_0\,$. For the construction of such extensions, see \cite[Theorem 2.14]{Ku21}.
Let $(L,v_p)$
be the henselization of this field. Since $(L,v_p)$ is henselian discretely valued of
characteristic $0$, it is a defectless field by \cite[Theorem 8.32]{KuHab}.
Alternatively, one can also take the completion in place of the henselization; as
the valuation is still discrete, this field is maximal and therefore a henselian
defectless field (see the discussion at the beginning of Section 4 in \cite{Ku58}).
\sn
For $\chara L=\chara L_0\,$: Take an element $z$ transcendental over $L_0\,$, the
$z$-adic valuation $v_z$ on $L_0(z)$, and $(L,v_z)$ to be the henselization of
$(L_0(z),v_z)$. Then $(L,v_z)$ is henselian discretely valued with residue field
$L_0\,$, and by \cite[Theorem 1.1]{Ku29}, it is a defectless field.
\end{proof}

If $w$ is a valuation on a field $L$ and $v_0$ a valuation on the residue field $Lw$,
then one can define the \bfind{composition} $w\circ v_0$ which is a valuation on $L$
with residue field $(Lw)v_0$ as follows. If $\cO_w$ is the valuation ring of $w$ on $L$
and $\cO_{v_0}$ is the valuation ring of $v_0$ on $Lw$, then the preimage of $\cO_{v_0}$
under the residue map $\cO_w\ni a\mapsto aw$ is a valuation ring contained in $\cO_w$
and we let $v:=w\circ v_0$ be the valuation associated with it (it is unique up to
equivalence of valuations). Then, modulo canonical isomorphisms, $v_0(Lw)$ can be
viewed as a convex subgroup of $vL$ and $wL$ can be viewed as the quotient $vL/v_0(Lw)$.

Now let $(L'|L,v)$ be a finite extension. Then $w$ extends to $L'$ and $v_0$ to $L'w$
in such a way that $v_0(L'w)$ is the convex hull of $v_0(Lw)$ and $wL'$ can be
identified with $vL'/v_0(L'w)$. In this situation, we have:
\begin{lemma}                             \label{ind}
If $wL'=wL$, then $(vL':vL)=(v_0(L'w):v_0(Lw))$.
\end{lemma}
\begin{proof}
The equality $wL'=wL$ means that the natural embedding
\[
vL/v_0(Lw)\>=\>wL\>\hookrightarrow\> wL'\>=\> vL'/v_0(L'w)
\]
is onto. It follows that $vL'=vL+v_0(L'w)$, which in turn implies our statement.
\end{proof}

\begin{lemma}
Take $(L_0,v_0)$ as in Lemma~\ref{L_0v_0}, and $(L,w)$ as in Lemma~\ref{L,w}.
Set $v:=w\circ v_0\,$. Then $(L,v)$ is algebraically maximal.
\end{lemma}
\begin{proof}
Since $(L,w)$ is henselian defectless, it is algebraically maximal. The composition
$w\circ v_0$ of two algebraically maximal valuations $w$ and $v_0$ is again
algebraically maximal (this is well known and the proof is straightforward).
\begin{comment}
Suppose that $(L'|L,v)$ is a nontrivial immediate algebraic extension. Then $vL'=vL$,
which implies that $wL'=wL$ and that $v_0(L'w)=v_0(Lw)$. Since $(L,w)$ is a
henselian defectless
field, we have $[L'w:Lw]=[L':L]$. Since $Lw=L_0$ is algebraically maximal under its
valuation $v_0\,$, it follows that $(v_0(L'w):v_0(Lw))>1$, which implies $(vL':vL)>1$,
or $[(L'w)v_0:(Lw)v_0]>1$, which implies $[L'v:Lv]>1$. This contradicts our assumption
that $(L'|L,v)$ is immediate.
\end{comment}
\end{proof}

\begin{lemma}
Let $(L,v)$ be as in the previous lemma. Then there are elements $a,b$ in the
separable-algebraic closure of $L$ such that $[L(a):L]=[L(b):L]=p$, $L(a)w=L_0(a_0)$,
and $L(b)w=L_0(b_0)$.
\end{lemma}
\begin{proof}
Take $c,d\in L$ such that $cw=a_0^p\in L_0$ and $dw=b_0^p\in L_0$. If $\chara L=0$,
then take $a$ to be a $p$-th root of $c$ and $b$ to be a $p$-th root of $d$. If
$\chara L=\chara L_0
=p$, then take $a$ to be a root of the polynomial $X^p-rX-c$ and $b$ to be a root of
the polynomial $X^p-rX-d$ for some $r\in L\setminus\{0\}$ with $wr>0$. Then in both
cases, $a$ and $b$ are separable over $L$ with $aw=a_0$ and $bw=b_0$. It follows that
\[
p\>\geq\>[L(a):L]\>\geq\>[L(a)w:Lw]\>\geq\>[L_0(a_0):L_0]\>=\>p\>.
\]
Hence equality holds everywhere, which proves that $[L(a):L]=p$ and $L(a)w=L_0(a_0)$.
The proof for $b$ in place of $a$ is similar.
\end{proof}

\pars
Now we are ready for the
\sn
{\it Proof of Theorem~\ref{amnotidl2}:} We shall prove that the valued field $(L,v)$
of the previous lemma has the properties stated in Theorem~\ref{amnotidl2}. As the
extensions $L(a)|L$ and $L(b)|L$ are separable, so is the extension $L(a,b)|L$.
Since $a_0,b_0\in L(a,b)w$, we have
\[
p^2\>\geq\>[L(a,b):L]\>\geq\>[L(a,b)w:Lw]\>\geq\> [L_0(a_0,b_0):L_0]\>=\>p^2\>,
\]
hence equality holds everywhere, showing that $[L(a,b):L]=p^2$ and $L(a,b)w=
L_0(a_0,b_0)$, so that $L(a,b)v=L_0(a_0,b_0)v_0\,$. On the other hand, $wL(a,b)=wL$
by the Fundamental Equality~(\ref{fundineq}) since $[L(a,b):L]=[L(a,b)w:Lw]$.
Further, by Proposition~\ref{L_0v_0}, $(v_0L_0(a_0,b_0):v_0L_0)=p$ and
$[L_0(a_0,b_0)v_0:L_0v_0]=1$. Hence
\[
(vL(a,b):vL)\>=\>(v_0(L(a,b)w):v_0(Lw))\>=\>(v_0L_0(a_0,b_0):v_0L_0)=p\>,
\]
where we have used Lemma~\ref{ind}, and
\[
[L(a,b)v:Lv]=[L_0(a_0,b_0)v_0:L_0v_0]\>=\>1\>.
\]
So the extension $(L(a,b)|L,v)$ has defect $p$.

\pars
Finally, $[L(a):L]=p$ and since $wL(a,b)=wL$, we also have $wL(a)=wL$.
Hence by Lemma~\ref{ind}, $(vL(a):vL)=(v_0(L(a)w):v_0(Lw))=
(v_0L_0(a_0):v_0L_0)=p$, showing that the
extension $(L(a)|L,v)$ is defectless. Since the defect is multiplicative, it
follows that $(L(a,b)|L(a),v)$ has defect $p$, which shows that $(L(a),v)$ is not
algebraically maximal. The same proof works for $b$ in place of $a$.
\qed

\bn

\end{document}